\documentclass[reqno,11pt]{amsart}

\title{Untwisting 3-strand torus knots}
\author{S.~Baader, I.~Banfield, L.~Lewark}

\usepackage{epsf}
\usepackage{graphics}
\usepackage{graphicx}
\usepackage{amssymb}
\usepackage{amsmath}
\date{}

\usepackage{microtype,hyperref,caption,mathtools}
\usepackage[nameinlink,nosort]{cleveref}
\let\cref\Cref
\hypersetup{colorlinks={true},linkcolor={black},citecolor={black},filecolor={black},urlcolor={black},pdfauthor={\authors},pdftitle={\shorttitle}}
\theoremstyle{plain}
\newtheorem{theorem}{Theorem}
\newtheorem{lemma}[theorem]{Lemma}
\newtheorem{proposition}[theorem]{Proposition}

\theoremstyle{definition}

\theoremstyle{remark}

\def\N{{\mathbb N}}

\def\gt{g_t}
\def\gs{g_s}

\begin{document}

\begin{abstract} We prove that the signature bound for the topological 4\nobreakdash-genus of 3-strand torus knots is sharp, using McCoy's twisting method.
We also show that the bound is off by at most $1$ for 4-strand and 6-strand torus knots, and
improve the upper bound on the asymptotic ratio between the topological 4-genus and the Seifert genus of torus knots from 2/3 to 14/27.  
\end{abstract}

\maketitle

\section{Introduction}

The braid group on three strands $B_3$ is generated by two elements $a,b$ satisfying the braid relation $aba=bab$. In this note, we are interested in the natural closure of the positive braid $(ab)^n$ in $S^3$, known as torus link of type $T(3,n)$. Whenever $n \in \N$ is a multiple of $3$, the link $T(3,n)$ has three components; otherwise it is a knot. The topological 4-genus $\gt(K)$ of a knot~$K \subset S^3$ is defined to be the minimal genus among all surfaces $\Sigma \subset D^4$, embedded in a locally flat way, with boundary $\partial \Sigma=K$. As with the smooth version of the 4-genus invariant, the topological 4-genus of knots~$K$ is bounded below by the signature invariant~\cite{Po}: $\gt(K) \geq |\sigma(K)|/2$. The same lower bound holds with the signature invariant replaced by the maximum value of the Levine-Tristram signature function outside of the set of roots of the Alexander polynomial $\Delta_K(t)$ of $K$
$$\widehat{\sigma}(K)\coloneqq\max_{\omega \in S^1\setminus \Delta_K^{-1}(0)} |\sigma_{\omega}(K)|.$$

\begin{theorem}
\label{3torus}
Let $n\geq 4$ be a natural number not divisible by three.
Then $$\gt(T(3,n))=\frac{\widehat{\sigma}(T(3,n))}{2} =  \left\lceil\frac{2n}{3}\right\rceil.$$
\end{theorem}

We believe that the equality $\gt=\widehat{\sigma}/2$ holds for a much larger class of torus knots, possibly for all. This can be seen as a topological counterpart of the local Thom conjecture, which states that the smooth 4-genus $\gs$ of torus knots coincides with their Seifert genus~\cite{KM,Ru2,Ra}. Unlike in the smooth case, where the hard part is finding suitable lower bounds, the difficulty in the topological case is figuring out genus-minimising surfaces (see~\cite{Ru1} and~\cite{BFLL} for first attempts in this direction). We will not see any of these surfaces. Rather, we will find a precise upper bound for the topological 4-genus via an operation called null-homologous twisting, which has recently received some attention \cite{In,Liv,MC,MC2}.
A null-homologous twist is an operation on oriented links that inserts a full twist into an even number $2m$ of parallel strands, $m$ of which point upwards, and $m$ of which point downwards (see e.g.\ \cref{fig:twist}).
Throughout this paper, we will use the term \emph{twist} for a null-homologous twist. The case of two strands corresponds to a simple crossing change. For a knot $K$, we define the \emph{untwisting number} $t(K)$ to be the minimal number of twists needed to transform $K$ into the trivial knot, as in~\cite{In}.
Relying on Freedman's disc theorem~\cite{Fr}, McCoy proved that the untwisting number is an upper bound for the topological 4-genus of knots~\cite{MC}. This is the tool we use to construct the genus-minimising surfaces in \cref{3torus}.

Let us take another look at the resemblance of the smooth and topological setting. Writing $s$ and $u$ for the Rasmussen invariant and unknotting number, respectively, it follows from the (smooth) local Thom conjecture that the inequalities
$$
s(K)/2 \leq \gs(K) \leq u(K),
$$
which hold for all knots $K$, become equalities for all torus knots:
$$
s(T(p,q))/2 = \gs(T(p,q)) = u(T(p,q)).
$$
We show that in the topological setting, in striking analogy, the inequalities
$$
\widehat{\sigma}(K)/2 \leq \gt(K) \leq t(K),
$$
which hold for all knots $K$, become equalities for all 3-strand torus knots:
$$
\widehat{\sigma}(T(3,n))/2 = t(T(3,n)) = \gt(T(3,n)).
$$
Thus in the topological setting, the untwisting number apparently takes the place that the unknotting number has in the smooth setting.

Untwisting might very well lead to the equality $g_t = \widehat{\sigma}/2 = t$ for all torus knots.
For the time being, we show that the equality is off by at most $1$ for torus knots with four and six strands.
\begin{proposition}\label{thm:46}
For all odd natural numbers $n \geq 3$,
$$
n \leq \frac{\widehat{\sigma}(T(4,n))}{2}
\leq \gt(T(4,n)) \leq t(T(4,n)) \leq \frac{2}{3}g(T(4,n)) + 2 = n + 1.$$
Moreover, for all natural numbers $n\geq 5$ coprime to $6$,
$$
\frac{3n+1}{2} \leq \frac{\widehat{\sigma}(T(6,n))}{2}
\leq
\gt(T(6,n)) \leq t(T(6,n)) \leq  \frac{3}{5}g(T(6,n)) + 3 = \frac{3n+3}{2}.
$$
\end{proposition}

McCoy also developed an induction scheme that allows him to estimate the asymptotic ratio between the topological 4-genus and the Seifert genus of torus knots~\cite{MC}:
$$\limsup_{p,q \to \infty} \frac{\gt(T(p,q))}{g(T(p,q))} \leq \frac{2}{3}.$$

\begin{theorem}
\label{asymptotic}
$$\limsup_{p,q \to \infty} \frac{\gt(T(p,q))}{g(T(p,q))} \leq \frac{14}{27} \approx 0.519.$$
\end{theorem}

The proof of \cref{3torus} makes use of a calculus for positive 3-braids introduced in~\cite{Ba}, which we present in the next section. \cref{sec:3,sec:46} contain the proofs of \cref{3torus} and \cref{thm:46}. The latter follows from the former by untwisting torus knots on four and six strands via torus knots on three strands.
\cref{asymptotic} follows from McCoy's induction scheme, which we briefly review in the last section.

\section{A calculus for positive 3-braids}

Let $k_1,k_2,\ldots,k_n$ be strictly positive integers. The positive braid
$$[k_1,k_2,\ldots,k_n]\coloneqq a^{k_1}ba^{k_2}b \cdots a^{k_n}b \in B_3$$
defines a link $L[k_1,k_2,\ldots,k_n]$, via its closure. For example, the torus link of type $T(3,n)$ can be written as $L[1,1,\ldots,1]$, where the number~1 appears~$n$ times. This notation is far from unique. The full twist on three strands can be written as
$$[1,1,1]=ababab=aabaab=[2,2].$$
The double full twist can be written as 
$$aab(abaaba)aab=aaabaaabaaab=[3,3,3].$$
In the first equality, we used the fact that the full twist $abaaba \in B_3$ commutes with all 3-braids. Adding another full twist to this, we obtain the following representative for the triple full twist:
$$aaabaaab(abaaba)aaab=aaabaaaabaaabaaaab=[3,4,3,4].$$
From here, we see that the operation
$$[\ldots,x,y, \ldots] \rightarrow [\ldots,x+1,3,y+1, \ldots]$$
corresponds to adding a full twist to a given positive braid on 3 strands. With this combinatorial calculus, we obtain the following family of positive braid presentations for iterated full twists on three strands.

\begin{lemma}
\label{fulltwists}
For all $k \in \N$:
\begin{enumerate}
\item $T(3,6k+9)=L[3,5^k,4,3,5^k,4]$,
\item $T(3,6k+12)=L[3,5^{k+1},3,4,5^k,4]$,
\end{enumerate}
where $5^k$ stands for a sequence $5,\ldots,5$ of length $k$.\qed
\end{lemma}

The proof is by induction on~$k$, starting at zero. A repeated application of the above move yields the desired sequence of presentations for increasing powers of the full twist: 
$$[3,5,3,4,4], \enspace [3,5,4,3,5,4], \enspace [3,5,5,3,4,5,4], \enspace [3,5,5,4,3,5,5,4], \, \ldots$$

\section{Untwisting torus knots with three strands}\label{sec:3}

\begin{lemma}
\label{untwisting}
For all $k \in \N$:
\begin{enumerate}
\item $t(T(3,3k+4)) \leq 2k+3$,
\item $t(T(3,3k+5)) \leq 2k+4$.
\end{enumerate}
\end{lemma}

\begin{proof} The two statements are obviously true for $k=0$, since the knots $T(3,4)$ and $T(3,5)$ can be unknotted by $3$ and $4$ crossing changes, respectively. Moreover, for all $k \in \N$, the two knots $T(3,3k+4)$ and $T(3,3k+5)$ are related by a single crossing change, so we only need to prove (1). We will do so by considering the three special cases 
$T(3,7)$, $T(3,10)$, $T(3,13)$ separately, and then the two families $T(3,6k+16)$, $T(3,6k+19)$.

The key observation is that the two braids $abbaabba$ and $bb$ are related by a sequence of two twists, as shown in \cref{fig:3braid}. Here the first arrow stands for a twist on four strands, while the second arrow is a simple crossing change.
\begin{figure}[ht]
\centering
$$\raisebox{-16mm}{\includegraphics[scale=1.2]{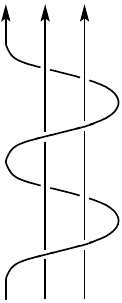}}
\enspace = \enspace
\raisebox{-16mm}{\includegraphics[scale=1.2]{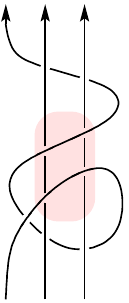}}
\enspace \rightarrow \enspace
\raisebox{-16mm}{\includegraphics[scale=1.2]{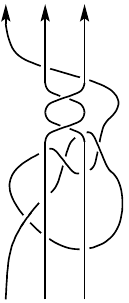}}
\enspace \rightarrow \enspace
\raisebox{-16mm}{\includegraphics[scale=1.2]{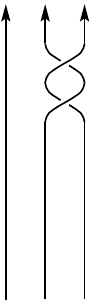}}
$$
\caption{Sequence of two twists. The first twist is on four strands, marked in red. The second twist is a crossing changes, untying the trefoil summand in the third drawing.}
\label{fig:3braid}
\end{figure}

As a consequence, the double full twist on three strands,
$$(ab)^6=abbaabbabbbb,$$
is related to the braid $b^6$ (and also to $a^6$) by a sequence of two twists.
For the first knot, $T(3,7)$, we turn the braid $(ab)^7$ into $a^7b$ by two twists, and into $ab$ by another three crossing changes, thus showing $t(T(3,7)) \leq 5$. In order to deal with the other two knots, we use the notation $A=a^{-1},B=b^{-1}$. We write
$$(ab)^{10}=(ab)^{12}BABA=(ab)^{12}ABAA=AB(ab)^{12}AA,$$
which transforms into $ABb^6a^6AA=Ab^5a^4$ by a sequence of four twists, and into $Aba^2$ by another three twists. The closure of the last braid is the trivial knot; this shows $t(T(3,10)) \leq 7$. For the knot $T(3,13)$, we observe that the braid
$$(AB)^5=(AB)^6ba=A^3BA^3BA^3Bba=A^3BA^3BA^2$$
represents the torus knot $T(3,-5)$. Therefore, we can write
$$(ab)^{13}=(ab)^{18}A^3BA^3BA^2=A^3(ab)^{12}BA^3(ab)^6BA^2,$$
which transforms into $A^3a^6b^6BA^3a^6BA^2=a^3b^5a^3BA^2$ by a sequence of six twists, and into $a^3baBA^2=a^2bA$ by another three twists. This shows $t(T(3,13)) \leq 9$. 

We now turn to the family of torus knots $T(3,6k+16)$. Using again $A=a^{-1},B=b^{-1}$, we write
$$T(3,6k+16)=L[ab(ab)^{6(2k+4)}(BA)^{6k+9}].$$
By \cref{fulltwists}, we have
$$(BA)^{6k+9}=A^3B(A^5B)^kA^4BA^3B(A^5B)^kA^4B,$$
which contains precisely $2k+4$ pure powers of $A$. We slide one double full twist $(ab)^6$ to the right of each power of $A$ and transform it into $a^6$ by a sequence of $2(2k+4)$ twists, in total. This leaves us with the braid
$$aba^3B(aB)^ka^2Ba^3B(aB)^ka^2B.$$
Sliding the half-twist $aba$ from the left to the middle yields
$$b^2A(bA)^kb^2abaBa^3B(aB)^ka^2B.$$
Then we transform the middle part $b^2abaBa^3B=b^3a^4B$ into the empty braid by three crossing changes. What remains is the braid $b^2Aa^2B$, whose closure is the trivial knot. Therefore $t(T(3,6k+16)) \leq 2(2k+4)+3=4k+11$, in accordance with statement (1).

The second family, $T(3,6k+19)$, works in complete analogy, using the expression
$$T(3,6k+19)=L[ab(ab)^{6(2k+5)}(BA)^{6k+12}]$$
and
$$(BA)^{6k+12}=A^3B(A^5B)^{k+1}A^3BA^4B(A^5B)^kA^4B.$$
The resulting intermediate braid, after a sequence of $2(2k+5)$ twists, is
$$aba^3B(aB)^{k+1}a^3Ba^2B(aB)^ka^2B=b^2A(bA)^{k+1}b^3Ab^2abaB(aB)^ka^2B.$$
Again, the middle part $b^3Ab^2ab$ transform into the empty braid by three crossing changes. The remaining braid is $b^2aB$, whose closure is the trivial knot.
This shows $t(T(3,6k+19)) \leq 2(2k+5)+3=4k+13$, in accordance with statement (1).
\end{proof}

As mentioned before, the inequalities
$$\widehat{\sigma}(K)/2 \leq \gt(K) \leq t(K)$$
hold for all knots $K$. To complete the proof of \cref{3torus}, we estimate the maximal Levine-Tristram signature $\widehat{\sigma}$ for the 3-strand torus knots. Let $K$ be a knot and consider the jump function $$\delta_K(x) = \lim_{s \to x^+} \sigma_{e^{2\pi i s}}(K) - \lim_{s \to x^-} \sigma_{e^{2\pi i s}}(K).$$ Litherland (see the comments after Proposition 1 in \cite{Lit}) notes that the discontinuities of the Levine-Tristram signature of the $T(p,q)$ torus knot occur precisely at the $x \in (0,1)$ satisfying $pqx \in \mathbb{Z}$ but $px \notin \mathbb{Z}$ and $qx \notin \mathbb{Z}$. For $x$ a discontinuity, let $pqx = pa + qb$ with $0 < a < q$. Litherland also shows that $\delta(x) = +2$ if $b < 0$ and $\delta(x) = -2$ if $b > 0$. If $x$ is the minimal discontinuity satisfying $x \geq \frac{1}{2}$, then $\widehat{\sigma}(K) \geq \sigma(K) + \delta_K(x)$. We distinguish the following cases.

\begin{description}
    \item[$T(3, 3k+4)$ for $k$ even] The periodicity of the ordinary signature, see Theorem 5.2 in \cite{GLM}, implies that $\widehat{\sigma}(T(3,3k+4)) \geq \sigma(T(3,3k+4)) = \sigma(T(3,4)) + 4k = 6 + 4k$.
    \item[$T(3, 3k+5)$ for $k$ even] Similarly to the previous case,\\ $\widehat{\sigma}(T(3,3k+5)) \geq 8 + 4k$.
    \item[$T(3, 3k+4) = T(3,6l+7)$ for $k = 2l+1$ odd] The minimal $x \geq \frac{1}{2}$ satisfying 
    $pqx \in \mathbb{Z}, px \notin \mathbb{Z}, qx \notin \mathbb{Z}$ is $x = \frac{9l+11}{3(6l+7)} $. Note that $pqx = 9l + 11 = 3(5l+6) + (6l + 7)(-1)$, which implies that $\delta(x) = +2$. The estimate $\widehat{\sigma} \geq \sigma + \delta(x)$ and the periodicity of the ordinary signature imply $\widehat{\sigma}(T(3, 3k+4)) \geq \sigma(T(3,3k+4) + 2 = \sigma(T(3,3(k+1)+1) = 4(k+1) + 2 = 4k+6$.
    \item[$T(3, 3k+5) = T(3,6l+8)$ for $k = 2l+1$ odd] Similar to the previous case, the minimal discontinuity $x \geq \frac{1}{2}$ is $x = \frac{9l+13}{3(6l+8)}$. From $pqx = 9l + 13 = 3(5l + 7) + (6l+8)(-1)$ it follows that $\delta(x) = 2$ and then $\widehat{\sigma}(T(3, 3k+4) \geq 4k+8$.
\end{description}
Therefore, for all $k \in \N$,
$$\widehat{\sigma}(T(3,3k+4)) \geq 4k+6, \enspace \widehat{\sigma}(T(3,3k+5)) \geq 4k+8,$$
and \cref{3torus} now follows.

\section{Untwisting torus knots with four and six strands}\label{sec:46}
To prove \cref{thm:46},
we essentially untwist torus knots with four and six strands to torus knots on three strands, and conclude using \cref{3torus}.

By a similar calculation as was presented at the end of the previous section for 3-strand torus knots,
one may prove that $\widehat{\sigma}(T(4,n)) \geq 2n$ and $\widehat{\sigma}(T(6,n)) \geq 3n + 1$.
For this, consider the signature (for $T(4,4k+3)$ and $T(6,k+5)$),
or the Levine-Tristram signature after the first jump after $1/2$ (for $T(4,4k+1)$), or after the second jump (for $T(6,6k+1)$). We note (without proof) that the stated inequalities for $\widehat{\sigma}$ are in fact equalities.

Now, let us show that $t(T(4,n)) \leq n + 1$.
Denote the standard Artin generators of the braid groups $B_4$ by $a,b,c$.
The crucial move is to transform three full twists on four strands, $(abc)^{12}$, into four full twists on three strands, $(bc)^{12}$ (or $(ab)^{12}$),
by four twist operations.
This can be seen by composing the braids in \cref{fig:4} with $(bc)^9$.

\begin{figure}[ht]
\centering
$$\raisebox{-16mm}{\includegraphics[scale=1.15]{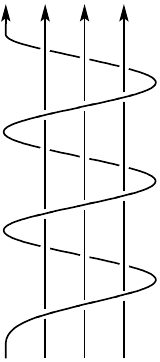}}
\enspace = \enspace
\raisebox{-16mm}{\includegraphics[scale=1.15]{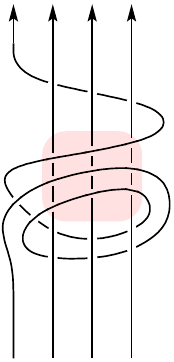}}
\enspace \rightarrow \enspace
\raisebox{-16mm}{\includegraphics[scale=1.15]{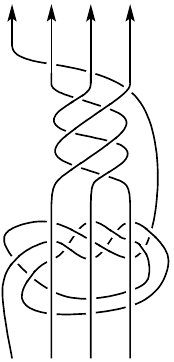}}
\enspace \rightarrow \enspace
\raisebox{-16mm}{\includegraphics[scale=1.15]{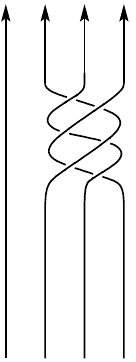}}
$$
\caption{Sequence of four twists. The first twist is on six strands, marked in red. The other three twists are crossing changes, untying the $T(3,4)$ summand in the third drawing.}
\label{fig:4}
\end{figure}

Hence, for $n = 12k + \varepsilon$, with $\varepsilon\in\{\pm 1,\pm 3,\pm 5\}$,
one may change
 $T(4,n) = L[(abc)^{n}]$
by $4k$ twists
into $L[(ab)^{12k}(abc)^{\varepsilon}]\eqqcolon K_{\varepsilon}$.
We now consider the possible values of $\varepsilon$ one by one, showing $t(T(4,n))\leq n + 1$ in each case, thus completing the proof of the first part of \cref{thm:46}.
\begin{itemize}
\item $K_1$ is in fact $T(3,12k+1)$, which may be untwisted by $8k+1$ twists, as established previously. In total, $t(T(4,n)) = n$.
\item Similarly $K_{-1} = T(3,12k-1)$, which may be untwisted by $8k$ twists, resulting in $t(T(4,n)) \leq n + 1$.
\item $K_3 = L[(ab)^{12k}(abc)^3] = L[(ab)^{12k}aba^2b^2ab] = T(3,12k+4)$, which may be untwisted by $8k + 3$ twists. In total $t(T(4,n)) = n$.
\item Similarly, $K_{-3} = T(3,12k - 4)$, which may be untwisted by $8k - 2$ twists, giving a total of $t(T(4,n)) \leq n + 1$.
\item $K_5$ can be transformed into $K_1$ by a twist on two strands and four crossing changes (cf.~\cref{fig:twist}), in total $t(T(4,n)) \leq n + 1$.
\item Similarly, $K_{-5}$ can be transformed into $T(3,12(k-1)+1)$ by five twists, resulting in $t(T(4,n)) \leq n + 1$.
\end{itemize}
This concludes the proof of the first half of \cref{thm:46}.

To show $t(T(6,n)) \leq (3n+3)/2$, denote the standard Artin generators of $B_6$ by $a,b,c,d,e$, respectively.
The full twist $(abcde)^6$ on six strands may be transformed by a single twist into
 $(ab)^6(de)^6$, see~\cref{fig:twist} for the analogous operation on four instead of six strands.
Applying this $k$ times to $T(6,6k\pm 1)$ yields the connected sum of two copies of $T(3,6k\pm 1)$,
which is finished off using \cref{3torus}. Summing up, the second half of \cref{thm:46} follows.

\section{Asymptotic genus ratio}
The key point in McCoy's induction scheme is that a positive full twist in a braid with $2n$ strands can be transformed into two parallel copies of positive double full twists in $n$ strands, with a single twist operation (see Lemma~13 in~\cite{MC}). This is shown in \cref{fig:twist}, for $n=2$, and was used in the previous section for $n=2$ and $n=3$. Similarly, a single twist operation transforms the torus knot $T(2n,2n+1)$ into the connected sum of two copies of the torus knot  $T(n,2n+1)$. This is seen by adding the word $\sigma_1 \sigma_2 \ldots \sigma_{2n-1}$ to both braids in the same figure.

\begin{figure}[ht]
\centering
$$\raisebox{-8mm}{\includegraphics[scale=1.6]{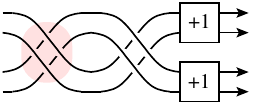}}
\enspace \rightarrow \enspace
\raisebox{-8mm}{\includegraphics[scale=1.6]{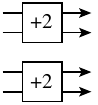}}
$$
\caption{Untwisting a full twist on four strands, marked in red. The numbers $+1(+2)$ stand for a (double) positive full twist.}
\label{fig:twist}
\end{figure}

When iterating this operation on successive powers of two, one gets $2/3$ as an upper bound for the asymptotic ratio $\gt/g$ for torus knots with increasing parameters. We will apply the same procedure, starting from braids with $3$ strands, successively multiplying the strand number by two:

\begin{enumerate}
\item $T(6,7)$ transforms into the disjoint union of two copies of $T(3,7)$ by one twist,
\item $T(12,13)$ transforms into the disjoint union of two copies of $T(6,13)$ by one twist, then into the disjoint union of four copies of $T(3,13)$ by 4 more twists,
\item $T(24,25)$ transforms into the disjoint union of two copies of $T(12,25)$ by one twist, then into the disjoint union of eight copies of $T(3,25)$ by $4(1+4)=20$ more twists,
\item $T(3 \cdot 2^k,3 \cdot 2^k+1)$ transforms into the disjoint union of $2^k$ copies of $T(3,3 \cdot 2^k+1)$ by a total number of $1+4+16+\ldots+4^{k-1}=1/3 \cdot (4^k-1)$ twists.
\end{enumerate}

By \cref{3torus}, the untwisting number of $T(3,3 \cdot 2^k+1)$ is of the order
$$2/3 \cdot (3 \cdot 2^k)=2^{k+1}.$$
We conclude that the untwisting number of $T(3 \cdot 2^k,3 \cdot 2^k+1)$ is bounded above by an expression of the order
$$1/3 \cdot (4^k-1)+2^k \cdot 2^{k+1} \approx (1/3+2) \cdot 4^k,$$
while its Seifert genus is of the order
$$1/2 \cdot (3 \cdot 2^k)^2=9/2 \cdot 4^k,$$
by the well-known genus formula $g(T(p,q))=1/2 \cdot (p-1)(q-1)$. In summary,
$$\limsup_{k \to \infty} \frac{\gt(T(3 \cdot 2^k,3 \cdot 2^k+1))}{g(T(3 \cdot 2^k,3 \cdot 2^k+1))} \leq \frac{1/3+2}{9/2}=\frac{14}{27}.$$
The existence of the more general upper bound,
$$\limsup_{p,q \to \infty} \frac{\gt(T(p,q))}{g(T(p,q))} \leq \frac{14}{27},$$
follows from a general principle on subadditive functions, see the proof of Theorem~5 in~\cite{MC}, or the paragraph preceding Theorem~2 in~\cite{BFLL}.
We are left with the strong belief that the ratio tends to $1/2$, in accordance with the asymptotic behaviour of the signature invariant:
$$\lim_{p,q \to \infty} \frac{\sigma(T(p,q))}{2g(T(p,q))}=\frac{1}{2}.$$

\bigskip
\noindent
Mathematisches Institut, Sidlerstr.~5, 3012 Bern, Switzerland

\smallskip
\noindent
Mathematisches Institut, Sidlerstr.~5, 3012 Bern, Switzerland

\smallskip
\noindent
Uni Regensburg, Fakult\"at f\"ur Mathematik, 93053 Regensburg, Germany

\newcommand{\myemail}[1]{\texttt{\href{mailto:#1}{#1}}}

\bigskip

\smallskip
\noindent
\myemail{sebastian.baader@math.unibe.ch}

\smallskip
\noindent
\myemail{ian.banfield@math.unibe.ch}

\smallskip
\noindent
\myemail{lukas@lewark.de}

\end{document}